\documentclass[12pt,leqno]{amsart}
\usepackage{amssymb,amsthm,amsmath,latexsym}
\newtheorem{theorem}{\sc Theorem}[section]

\newtheorem{lemma}[theorem]{\sc Lemma}

\newtheorem{problem}[theorem]{\sc Problem}

 \title[Nilpotency Criterion for Nilpotent Residual]{A Sufficient Condition for Nilpotency of the Nilpotent Residual of a Finite Group}

\author{Agenor Freitas de Andrade} 
\address{Instituto Federal de Educa\c{c}\~ao, Ci\^{e}ncia e Tecnologia de Goi\'as, Luzi\^{a}nia - GO, 72811-580, Brazil}
\email{agenor.andrade@ifg.edu.br}

\author{Alex Carrazedo Dantas}
\address{Universidade Tecnol\'ogica Federal do Paran\'a, Guarapuava - PR, 85053 - 525, Brazil}
\email{alexcdan@gmail.com}

\begin{document}
\begin{abstract}
Let $G$ be a finite group with the property that if $a,b$ are powers of $\delta_1^*$-commutators such that $(|a|,|b|)=1$, then $|ab|=|a||b|$. We show that $\gamma_{\infty}(G)$ is nilpotent. 
\end{abstract}
\maketitle
\section{Introduction}
All groups considered in the present paper are finite. The last term of the lower central series of a group $G$ is called the nilpotent residual. It is usually denoted by $\gamma_\infty(G)$. The lower Fitting series of $G$ is defined by $D_0(G)=G$ and $D_{i+1}(G)=\gamma_\infty(D_i(G))$ for $i=0,1,2,\dots$.


Here the notation $|x|$ stands for the order of the element $x$ in a group $G$. The following criterion of nilpotency of a finite group was established by Baumslag and Wiegold in \cite{bw}.

\begin{theorem}\label{1}
Let $G$ be a finite group in which $|ab|=|a||b|$ whenever the elements $a,b$ have coprime orders. Then $G$ is nilpotent.
\end{theorem}

The following result was obtained by Bastos and Shumyatsky \cite{bs}.

\begin{theorem}\label{2}
Let $G$ be a finite group in which $|ab|=|a||b|$ whenever the elements $a,b$ are commutators of coprime orders. Then $G^{'}$ is nilpotent.
\end{theorem}

In \cite{bs} the following question is posed.
Let $w$ be a commutator word and $G$ a finite group with the property that if $a,b$ are $w$-values of coprime order, then $|ab|=|a||b|$. Is then the verbal subgroup $w(G)$ nilpotent?
%
A coprime commutator in  a group $G$ is a commutator $[x,y]$ with $|x|$ and $|y|$. The main result of this article is to obtain a similar result of the Theorems \ref{1} and \ref{2}  for the powers of coprime commutators.

\begin{theorem}\label{teo1}
Let $G$ be a finite group and set $$X = \{[a,b]^k | \, a, b \in G, (|a|,|b|)=1, k \in \mathbb{Z}\}.$$
Assume that for all $x,y \in X$ with $|x|$ and $|y|$ coprimes, we have $|xy| = |x||y|$. Then $\gamma_\infty(G)$ is nilpotent. In particular, $G$ has Fitting height at most $2$.
\end{theorem}
The coprime commutators $\gamma_k^*$ and $\delta_k^*$ have been introduced in \cite{shumyatsky} by Shumyatsky as a tool to study properties of finite groups that can be expressed in terms of commutators of elements of coprime orders. The definition of $\delta_k^*$ goes as follows (for more details see \cite{shumyatsky}). Every element of $G$ is a $\delta_0^*$-commutator. Now, let $k\geq 1$, let $T$ be the set of all elements of $G$ that are powers of $\delta_{k-1}^*$-commutators. The element $g$ is a $\delta_k^*$-commutator if there exist $a,b\in T$ such that $g=[a,b]$ and $(|a|,|b|)=1$. It was shown in \cite{shumyatsky} that for every $k\geq 0$ the subgroup generated by $\delta_k^*$-commutators is precisely $D_k(G)$.



Remark that the set $X$ in the Theorem \ref{teo1} is precisely the set of generators of $D_1(G)$, that is, $D_1(G) = \langle X \rangle$. Consequently we could ask the following question, which extends Theorem \ref{teo1}.

\begin{problem}\label{teo2}
Let be $k$ a non-negative integer and $G$ a finite group in which $|ab|=|a||b|$ whenever the elements $a,b$ are (powers of) $\delta_k^*$-commutators of coprime orders. It is true that the subgroup $D_k(G)$ is nilpotent? 
\end{problem}

\section{Preliminaries}
Throughout the remaining part of this short note, $G$ denotes a finite group satisfying the hypothesis of Theorem \ref{teo1}. For simplicity, we denote by $X$ the set of all powers of $\delta_1^*$-commutators of $G$. As usual, $O_\pi(M)$ stands for the maximal normal $\pi$-subgroup of a group $M$. The Fitting subgroup of $M$ is denoted by $F(M)$. We start with the following known lemma.

\begin{lemma}[\cite{issacs}, Lemma 4.28]\label{2a}
Let $A$ and $G$ be finite groups. Let $A$ act via automorphisms on $G$, and suppose that $(|A|,|G|)=1$. Then $G = [G,A]C_G(A)$.
\end{lemma}

Analogously to Lemma 3 of \cite{bs}, we have:

\begin{lemma}\label{3}
Let $x \in X$ and $N$ be a subgroup normalized by $x$. If $(|x|, |N|)=1,$ then $[x,N]=1$
\end{lemma}

\begin{proof}
Let $y \in N$. Then $[x, y]$ is a $\delta_{1}^{*}$-commutator and $x[x, y] = x^{y}$. Thus $|x[x, y]| = |x||[x, y]|$ and $|x[x, y]| = |x|$, this is, $[x, y] = 1$.
\end{proof}

The following result is the version of the well-know Focal Subgroup Theorem for $\delta_k^*$-commutators of a finite group.

\begin{lemma}[\cite{cpa}, Lemma 2.6]\label{4}
Let $k\geq 0$ and let $G$ be a finite soluble group of order $p^an$, where $p$ is a prime and $n$ is not divisible by $p$, and let $P$ be a Sylow $p$-subgroup of $G$. Then $P \cap \delta_k^*(G)$ is generated by $n$th powers of $\delta_k^*$-commutators lying in $P$.
\end{lemma}

The next theorem will also be useful. Recall the the Fitting height $h=h(G)$ of a finite soluble group $G$ is the minimal number $h$ such that $G$ possesses a normal series all of whose quotients are nilpotent.

\begin{theorem}[\cite{shumyatsky}, Theorem 2.6]\label{5}
Let $G$ be a finite group and $k$ a positive integer. We have $D_k(G) = 1$ if and only if $G$ is soluble with Fitting height at most $k$.
\end{theorem}

Our first result ensures that the theorem is valid for soluble groups. Recall that a finite group $G$ is metanilpotent if and only if $\gamma_{\infty}(G)$ is nilpotent.

\begin{lemma} \label{6}
If $G$ is soluble, then $\gamma_{\infty}(G)$ is nilpotent.
\end{lemma}
\begin{proof}
Recall that $D_1(G) = \gamma_\infty(G)$. Suppose that $|G|>1$. As $G$ is soluble we have that $|G'| < |G|$. Then by induction on $|G|$ we can assume that $D_1(G')$ is nilpotent. Thus $D_{2}(G') = 1$ and $G'$ has Fitting height $\leq 2$. So $G$ has Fitting height $\leq 3$. Then $D_{3}(G) = 1$ (Theorem \ref{5}). This implies that $D_{2}(G)$ is nilpotent. In particular $D_1(G)$ is metanilpotent. By simplicity, denote $D_{2}(G)$ by $H$. Therefore if $P$ is a Sylow $p$-subgroup of $H$ then $P$ is characteristic in $H$. Obviously $H$ is characteristic in $G$ and so we can conclude that $P$ is characteristic in $G$. In particular, $P \unlhd G$. Suppose that $D_1(G)$ is not a $p$-group. Let $Q$ be a Sylow $q$-subgroup of $D_1(G)$, with $q \neq p$. By Lemma \ref{4} $Q$ is generated by $n$th powers of $Q \cap X$, where $|D_1(G)| = q^an$ and $(q,n)=1$. Thus using Lemma \ref{3} we have that $[P,x]=1$ for all generators $x$ of $Q$. Consequently $QP$ is normal in $QH$. Since $D_1(G)$ is a metanilpotent we have $QH \unlhd D_1(G)$ and then it follows that 
$$(QH)^G = \displaystyle \prod_{g \in G}(QH)^g$$
is nilpotent, because every $(QH)^g$ is nilpotent and normal in $D_1(G)$. Then we conclude that $Q \leq F(G)$. If $|\pi(H)| > 1$ then all Sylow $q$-subgroups of $D_1(G)$ are normal and so the result follows. If $H$ is a $p$-group, then the Sylow $p$-subgroup of $D_1(G)$ is normal (because $D_1(G)$ is metanilpotent) and, applying the same argument now to the Sylow subgroups of $D_1(G)$, we can conclude that they are all normal, as desired.
\end{proof}

Recall that a group $M$ is said to be quasisimple if $M=M'$ and $M/Z(M)$ is simple. We are now ready to complete the proof of Theorem \ref{teo1}.
\section{The proof of Theorem \ref{teo1}}
\begin{proof}
We affirm that $G$ is soluble. Indeed, we assume otherwise. Let $G$ be a counter-example of minimal order. So all proper subgroups in $G$ are soluble and we can assume that $G = D_1(G)$. Let $R$ be the soluble radical in $G$. By Lemma \ref{6} $D_1(R)$ is nilpotent. Suppose that $G$ is nonsimple and $R \neq 1$. Let $P$ to be Sylow $p$-subgroup of $F(G)$. Let $q \neq p$ be a prime number and denote by $T$ the subgroup generated by all $\delta_1^*$-commutators whose order is a power of $q$. By Lemma \ref{4}, all $q$-subgroups of $D_1(G)$ are contained in $T$. Then $G/T$ is a $p$-group. Since $G=D_1(G)$, we conclude that $G=T$ (indeed, if $G \neq T$, then $T$ is soluble and this implies that $G$ is soluble, a contradiction). Now if $x$ is a $\delta_1^*$-commutator of a Sylow $q$-subgroup of $G$, with $q \neq p$, we conclude by Lemmas \ref{3} and \ref{4} that $[P,x]=1$. Then $C_{G}(P)$ is a normal subgroup of $G$, and $G/C_{G}(P)$ is a $p$-group, that is, $G = C_{G}(P)$ and $P \leq Z(G)$. Consequently $F(G) \leq Z(G)$. Since $R$ is soluble it follows by Lemma \ref{6} that $D_1(R)$ is nilpotent. Clearly $D_1(R) \unlhd G$. Then $D_1(R) \leq F(G) \leq Z(G)$. In particular, $D_1(R) \leq Z(R)$. Thus $R$ is nilpotent and therefore $R = Z(G)$. Thus, our group $G$ is quasisimple. By the Feit-Thompson Theorem,  $G$ has a Sylow $2$-subgroup $S$ which is not contained in the centre of $G$. By the Frobenius Theorem \cite[Theorem 7.4.5]{gorenstein} there exists a $2$-subgroup $L$ of $S$ such that $N_G(L)$ contains an element $b$ of odd order such that $[b,t] \notin Z(G)$, for all $t \in L$. Then $[b,t]$ is of $2$-power order. Let $k$ such that $z=[b,t]^k \notin Z(G)$ and $z^2 \in Z(G)$. We have $z \in X$. By the Baer-Suzuki Theorem \cite[Theorem 3.8.2]{gorenstein} there exists $h \in G$ such that $\langle z,z^h \rangle Z(G)/Z(G)$ is a dihedral group of order $2n$ for $n$ odd. Let $J = \langle z z^h \rangle Z(G)$. Then $J/Z(G)$ is inverted by $z$ and $J$ is nilpotent. However Lemma \ref{3} shows that $z$ centralizes $O_{2'}(J)$ a contradiction. The proof is complete.  

\end{proof}

\end{document}